\documentclass{article}

\usepackage[margin=1in]{geometry}

\usepackage{amsmath,amsthm,amssymb}
\usepackage{graphicx,wrapfig}
\usepackage{mathrsfs,upgreek}
\usepackage{booktabs}
\usepackage{commath}
\usepackage[colorlinks,linkcolor=linkcol]{hyperref}
\usepackage[nameinlink,noabbrev,capitalize]{cleveref}
\usepackage{subfigure}
\usepackage[inline]{enumitem}
\usepackage{calc}
\usepackage{cite}
\usepackage[all]{hypcap}
\usepackage{tikz}
\usetikzlibrary{patterns}
\usetikzlibrary{positioning}
\usetikzlibrary{matrix}
\crefalias{subequation}{equation}

\newcommand{\R}{ {\mathbb R} }
\newcommand{\Z}{ {\mathbb Z} }

\newcommand{\F}{ {\mathbb F} }
\newcommand{\va}{ \mathbf{a} }

\newcommand{\vu}{ \mathbf{u} }
\newcommand{\vv}{ \mathbf{v} }
\newcommand{\vw}{ \mathbf{w} }
\newcommand{\vx}{ \mathbf{x} }
\newcommand{\vy}{ \mathbf{y} }

\newcommand{\affil}[1]{\textit{#1}}
\newcommand{\email}[1]{\texttt{#1}} 
\newcommand{\currentrestr}{%
  \,\raisebox{-.127ex}{\reflectbox{\rotatebox[origin=br]{-90}{$\lnot$}}}\,%
}

\DeclareMathOperator{\lip}{Lip}
\newcommand{\boundary}{\partial}

\newenvironment{subproof}{\noindent{\it Proof of claim.}}{}
\theoremstyle{plain}
\newtheorem{theorem}{Theorem}[section]
\newtheorem{corollary}[theorem]{Corollary}
\newtheorem{conjecture}[theorem]{Conjecture}
\newtheorem{lemma}[theorem]{Lemma}
\newtheorem{subclaim}{Claim}[theorem]
\theoremstyle{definition}

\theoremstyle{remark}

\newcommand{\spt}{\operatorname{spt}}
\newcommand{\perimeter}{\operatorname{perimeter}}

\newcommand{\diam}{\operatorname{diameter}}
\newcommand{\inrad}{\operatorname{inradius}}

\newcommand{\mass}{\operatorname{M}}
\newcommand{\skel}[1]{\operatorname{skel}_{#1}}

\newcommand{\current}[1]{\mathcal{D}_{#1}}
\newcommand{\form}[1]{\mathcal{D}^{#1}}

\title{{Flat Norm Decomposition of Integral Currents}}

\author{Sharif~Ibrahim\thanks{\affil{Department of Mathematics, Washington State University, Pullman, WA 99164-3113},\newline
\email{\hspace*{0.15in}\{math.msfn,bkrishna,vixie\}@\{sharifibrahim.com,math.wsu.edu,speakeasy.net\}}}\,\,\footnote{Corresponding author}\hspace*{0.3in}
  Bala~Krishnamoorthy\footnotemark[1]\hspace*{0.3in}
  Kevin R.~Vixie\footnotemark[1]
}

\begin{document}

\maketitle

\definecolor{linkcol}{rgb}{0,0,0.5}
\begin{abstract}
{\em Currents} represent generalized surfaces studied in geometric measure theory. 
They range from relatively tame integral currents representing oriented compact manifolds with boundary and integer multiplicities, to arbitrary elements of the dual space of differential forms.
The {\em flat norm} provides a natural distance in the space of currents, and works by decomposing a $d$-dimensional current into $d$- and (the boundary of) $(d+1)$-dimensional pieces in an optimal way.

Given an integral current, can we expect its flat norm decomposition to be integral as well?
This is not known in general, except in the case of $d$-currents that are boundaries of $(d+1)$-currents in $\R^{d+1}$ (following results from a corresponding problem on the $L^1$ total variation ($L^1$TV) of functionals).
On the other hand, for a discretized flat norm on a finite simplicial complex, the analogous statement holds even when the inputs are not boundaries.
This simplicial version relies on the total unimodularity of the boundary matrix of the simplicial complex -- a result distinct from the $L^1$TV approach.

We develop an analysis framework that extends the result in the simplicial setting to one for $d$-currents in $\R^{d+1}$, provided a suitable triangulation result holds.
In $\R^2$, we use a triangulation result of Shewchuk (bounding both the size and location of small angles), and apply the framework to show that the discrete result implies the continuous result for $1$-currents in $\R^2$.
\end{abstract}

\section{Introduction}
In geometric measure theory, currents represent a generalization of oriented surfaces with multiplicities.
Currents were developed in the context of Plateau's problem, and have found application in isoperimetric problems and soap bubble conjectures\cite{Morgan2008}.

Given a $d$-dimensional current $T$, we can consider decompositions $\,T = X + \boundary S\,$ where $X$ is a $d$-dimensional current and $S$ is a $(d+1)$-dimensional current.
Over all such decompositions, the minimum total mass (volume) of the two pieces (i.e., $\mass(X) + \mass(S)$) is the flat norm $\F(T)$.
More recently, the $L^1$TV functional (introduced in the form most relevant to us by Chan and Esedo\=glu\cite{ChEs2005}) was shown to be related to the flat norm\cite{MoVi2007}.
This connection suggested the flat norm with scale (yielding the objective $\mass(X) + \lambda \mass(S)$ for any fixed scale $\lambda \geq 0$), and a geometric interpretation for the optimal decompositions: varying $\lambda$ controls the scale of features isolated in the decomposition.

One natural question: must currents in a particular regularity class have an optimal flat norm decomposition in the same class\cite{Federer1969,Wh1999}?
We consider the class of integral currents.
The $L^1$TV connection shows this result is true for boundaries of codimension $1$ (i.e., $d$-currents that are boundaries in $\R^{d+1}$), since the $L^1$TV functional applied to binary (or step function) input is known to have binary (step function) minimizers\cite{ChEs2005}.
We previously studied\cite{IbKrVi2013} a discrete version of the flat norm defined on a simplicial complex $K$, where finding the simplicial flat norm of a $d$-current represented by a $d$-chain of $K$ amounts to solving an integer linear optimization problem. 
The $(d+1)$-boundary matrix of $K$ embedded in $\R^{d+1}$ is guaranteed to be totally unimodular\cite[Theorem 5.7]{DeHiKr2011}. 
This property implies the integrality of the simplicial flat norm decomposition of currents that are not necessarily boundaries.
Notice that the result in the setting of simplicial complexes is distinct from the $L^1$TV approach.

Natural applications of the flat norm often involve integral currents.
For example, consider the space of handwritten signatures (or, similarly, topographical maps) along with the flat norm.
Any two signatures can be naturally represented as integral currents and the flat norm can compute a distance between them.
In addition to this distance, the flat norm also provides an optimal decomposition which determines the optimal way to turn one signature into the other.
These decompositons are most useful when they can be interpreted in the same way as the inputs (i.e., they are integral currents as well).

\noindent {\bfseries Our Contributions:}
In the present work, we develop an analysis framework to bridge the gap between the continuous and the discrete cases. 
Assuming a suitable triangulation result, 
our framework allows us to drop the requirement that integral $d$-currents in $\R^{d+1}$ be boundaries in order to have a guaranteed integral optimal decomposition.
We prove this necessary triangulation result in $\R^2$ using Shewchuk's Terminator algorithm\cite{Sh2002} for subdividing planar straight line graphs.
This algorithm simultaneously bounds the smallest angles in the complex and tells us where they can occur, allowing us to tailor a simplicial complex to a given set of input currents.
We then obtain a simplicial deformation theorem with constant bounds for these currents and the simplicial complex, ensuring the sequence of approximating discretized problems are well-behaved, and solve the continuous problem in the limit.
Assuming a suitable triangulation result for higher dimensions (see \cref{con:boundreg}), we show that codimension-$1$ integral currents have an integral optimal flat norm decomposition (\cref{thm:icmainresult}).

\noindent {\bfseries Related Work:}
Several related questions were considered by Almgren--- see his {\em Unfinished Work} as reported by White\cite{Wh1998}.
In particular, Almgren considered the question:  if $2T_i$ is a sequence of integral flat chains that converge in the integral flat topology, must the sequence $T_i$ also converge?
As White reported\cite{Wh1998}, Almgren ``... seemed to use every weapon in his arsenal, including his enormous $(m-2)$ regularity paper''.
Yet, little progress has since been reported on this problem.

For the related problem of least area with a given boundary (which can be considered as the flat norm problem with $X$ constrained to be empty), counterexamples of Young\cite{Yo1963}, White\cite{Wh1984}, and Morgan\cite{Mo1984} provide instances in which the minimizer is not integral for a given integral boundary.
R.~Young provides an excellent illustration of how these examples work\cite[Fig.~1]{Yo2013} and, by bounding the nonorientability of cycles, limits how much ``cheaper by the dozen'' such integral minimizers can be.

These counterexamples have codimension $3$ (the inputs are $1$-dimensional curves in $\R^4$) and imply codimension 3 counterexamples for the flat norm integral decomposition question as well (see \cref{sec:leastarea}).

\subsection{Definitions}

Let $\form{d}$ be the set of $C^\infty$ differentiable $d$-forms with compact support.
The set of $d$-currents (denoted $\current{d}$) is the dual space of $\form{d}$ with the weak topology.

Currents have mass and boundary that correspond (for rectifiable currents, at least) to one's intuition for what these should mean for $d$-dimensional surfaces in $\R^n$ with care taken to respect orientation and multiplicities.
The mass of a $d$-current $T$ is formally given by $\sup_{\phi \in \form{d}} \{T(\phi) \mid \norm{\phi} \leq 1\}$ and the boundary is defined when $d \geq 1$ by $\boundary T(\psi) = T(\dif \psi)$ for all $\psi \in \form{d-1}$.
When $T$ is a 0-current, we let $\boundary T = 0$ as a 0-current.
%This last convention is not universal, but allows us to simplify our presentation slightly.
The boundary operator on currents is linear and nilpotent (i.e., $\boundary \boundary T = 0$ for any current $T$), inheriting these properties from exterior differentiation of forms (which are linear and satisfy $\dif \dif \phi = 0$).

{\it Normal $d$-currents} have compact support and finite mass and boundary mass (i.e., $\mass(T) + \mass(\boundary T) < \infty$).
The set $\mathcal{R}_d$ denotes the {\it rectifiable $d$-currents} and contains all currents with compact support that represent oriented rectifiable sets with integer multiplicities and finite mass.
That is, sets which are almost everywhere the countable union of images of Lipschitz maps from $\R^d$ to $\R^n$.
Lastly, the set $\mathcal{I}_d$ represents {\it integral $d$-currents} and contains all currents that are both rectifiable and normal (formally, it is the set of rectifiable currents with rectifiable boundary, but this definition is equivalent by the closure theorem\cite[4.2.16]{Federer1969}).

The flat norm of a current $T$ is given by 
\[
\F(T) = \min \{\mass(X) + \mass(S) \mid T = X + \boundary S, X \in \mathcal{E}_d, S \in \mathcal{E}_{d+1}\}
\]
where $\mathcal{E}_d$ is the set of $d$-dimensional currents with compact support (see \cref{fig:icflatnorm}).
The Hahn-Banach theorem guarantees this minimum is attained\cite[p. 367]{Federer1969} so it makes sense to talk about particular $X$ and $S$ as a flat norm decomposition of $T$ (note, however, that the decomposition need not be unique).

\begin{figure}[hb!]  
\centering
\includegraphics[scale=0.7, trim=1.25in 8.3in 2.8in 1in, clip]{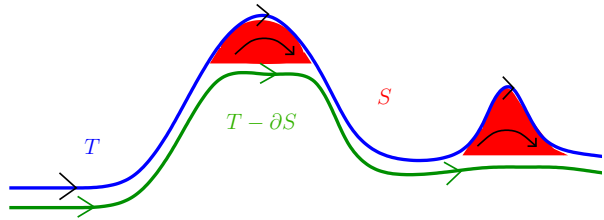}
\caption[Flat norm decomposition]{The flat norm decomposes the 1D current $T$ into (the boundary of) a 2D piece $S$ and the 1D piece $X = T - \boundary S$. The current $X$ is shown slightly separated from the input current $T$ for clearer visualization.}
\label{fig:icflatnorm}
\end{figure}

For two currents, the flat distance between them is given by $\F(T, P) = \F(T-P)$.
This definition is useful because it is robust to small additions and perturbances (e.g., noise) and reflects when currents are intuitively close.
For example, given a current $T$ representing a unit circle in $\R^2$ and an inscribed $n$-gon $T_n$ (both oriented clockwise, see \cref{subfig:icflatdistance1}), one would like $T_n$ to converge to $T$ in some sense as $n \to \infty$ which the flat norm accomplishes (contrast with the mass norm $\mass(T_n - T)\rightarrow 4\pi$).

\begin{figure}[ht!]
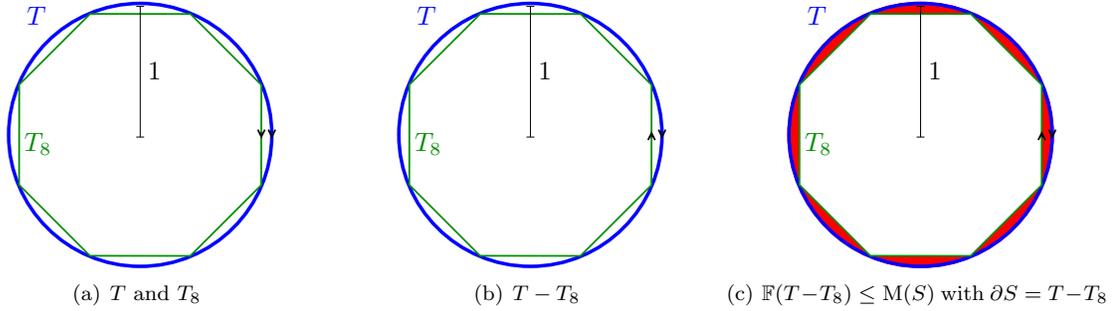

\begin{center}
\subfigure[$T$ and $T_8$]{\label{subfig:icflatdistance1} \makebox[\linewidth/10*3][c]{\input{polygonexampleb1.pspdftex}}}
\subfigure[$T - T_8$]{\label{subfig:icflatdistance2} \makebox[\linewidth/10*3][c]{\input{polygonexampleb2.pspdftex}}}
\subfigure[$\F(T-T_8) \leq \mass(S)$ with $\boundary S = T - T_8$]{\label{subfig:icflatdistance3} \makebox[\linewidth/10*3][c]{\input{polygonexampleb3.pspdftex}}}
%\subfigure[]{\subimport{figs}{polygonexampleb1.pspdftex}}
%\subfigure[]{\input{figs/polygonexampleb2.pspdftex}}
%\subfigure[]{\input{figs/polygonexampleb3.pspdftex}}
\end{center}
\caption[Flat distance example]{The flat norm indicates the unit circle $T$ and inscribed $n$-gon $T_n$ are close because the region they bound has small area.}
\end{figure}

The flat norm can be discretized in a natural sense.
Given a simplicial $(d+1)$-complex $K$ and a $d$-chain $T$ on $K$, the simplicial flat norm\cite{IbKrVi2013} of $T$ on $K$ is denoted by $\F_K(T)$ and is defined analogously except that $X$ and $S$ are restricted to be chains on $K$.
%\[
%\F_K^\lambda(T) = \min_{\vs \in \Z^n} \cbr{\sum_{i=1}^m \operatorname{V}_d(\sigma_i)\abs{x_i} + \lambda \sum_{j=1}^n \operatorname{V}_{d+1}(\tau_j)\abs{s_j} \mid \vx = \vt - [\boundary_{d+1}]\vs, \vx \in \Z^m}
%\]

\subsection{Overview}
We try to broadly follow the standard notion which expresses the continuous problem as a limit of discrete problems for which the result holds.
But the challenge is in working out the details.
\cref{thm:simpint} tells us that the {\em simplicial} flat norm of an integral chain in codimension $1$ has an optimal integral current decomposition; by the compactness theorem from geometric measure theory, the limit of these decompositions is also integral.

In order to show that an integral current $T$ has integral flat norm decomposition, we therefore find suitable simplicial approximations to $T$ and take the limit of their simplicial flat norm decompositions to obtain an integral decomposition for $T$.

We must also show that this decomposition achieves the flat norm value for $T$ (that is, express $T$ using integral currents in such a way that it remains an optimal flat norm decomposition).
Using the compactness theorem, this result is immediate if our simplicial approximations to $T$ have simplicial flat norm values that converge to the flat norm of $T$, but this convergence in flat norm values may not hold (see \cref{fig:flatnormconv}).
We wish to show 
\begin{align}
\label{eq:fteqgoal}
\lim_{\delta \downarrow 0}\F_{K_\delta}(P_\delta) &= \F(T)
\end{align}
 where $P_\delta$ is a simplicial approximation to $T$ on some complex $K_\delta$ with $\F(P_\delta - T) < \delta$.

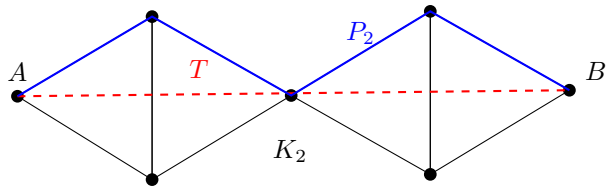
\begin{figure}[hb]
\medskip
\centering
\begin{tikzpicture}[scale=1]
    \tikzstyle{point}=[circle,thick,draw=black,fill=black,inner sep=0pt,minimum width=4pt,minimum height=4pt]
    \tikzstyle{current}=[thick]
%    \node (a)[point] at (0,0) {};
%    \node (b)[point] at (3,0) {};
%    \node (c)[point] at (2,2) {};

%    \begin{scope}[yshift=2cm]
%    \node (d)[point] at (1,1) {};
%    \node (e)[point] at (0,2) {};
%    \node (f)[point] at (4,2) {};
%    \end{scope}
    
    \node (p)[point,label={[label distance=0cm]90:$A$}] at (0,0) {};
    \node (K2)[below right = 1cm and 3.5cm,label={[label distance=0cm]90:$K_2$}] {};

    \node[above right = 1cm and 1.73cm] (t1)[point] {};
    \node[below = 2cm of t1] (t2)[point] {};
    \node[above right = 1cm and 1.73cm of t2] (t3)[point] {};
    \node[above right = 1cm and 1.73cm of t3] (t4)[point] {};
    \node[below = 2cm of t4] (t5)[point] {};
    \node[above right = 1cm and 1.73cm of t5] (q)[point,label={[label distance=0cm]5:$B$}] {};

     \draw (p.center) -- (t1.center) -- (t2.center) -- cycle;
     \draw (t1.center) -- (t2.center) -- (t3.center) -- cycle;
     \draw (t4.center) -- (t5.center) -- (t3.center) -- cycle;
     \draw (t4.center) -- (t5.center) -- (q.center) -- cycle;
     \draw[current,color=red,dashed] (p.center) -- node[label={[label distance=1cm]176:$T$}] {} (q.center);
     \draw[current,color=blue] (p.center) -- (t1.center) -- (t3.center) -- node[label={[label distance=0cm]88:$P_2$}] {} (t4.center) -- (q.center);

%    \draw[pattern=north east lines] (a.center) -- (p.center) -- (b.center) -- cycle;
%    \draw[pattern=north west lines] (a.center) -- (p.center) -- (c.center) -- cycle;
%    \draw[pattern=vertical lines]   (b.center) -- (p.center) -- (c.center) -- cycle;
%    \draw[pattern=dots] (d.center) -- (e.center) -- (f.center) -- cycle;
%    \draw (p.center) -- (d.center);
\end{tikzpicture}
\caption[Simplicial flat norm need not converge to continuous flat norm]{A sequence of simplicial chains that converges in the flat norm (i.e., $P_n \rightarrow T$) need not have convergent simplicial flat norm values (i.e., $\F_{K_n}(P_n) \rightarrow \F(T)$ need not hold).
The current $T$ is the line segment from $A$ to $B$ (shown dashed), the complex $K_n$ is the arrangement of $2n$ equilateral triangles of appropriate size stretching from $A$ to $B$ and $P_n$ is the top chain from $A$ to $B$ on $K_n$.
Clearly, $\F(T-P_n) \rightarrow 0$ but $\F_{K_n}(P_n) = \frac{2}{\sqrt{3}}\F(T) \not\rightarrow \F(T)$.}
\label{fig:flatnormconv}
\end{figure}

This goal prevents us from simply using the simplicial deformation theorem to obtain $P_\delta$, since we may end up with the situation illustrated in \cref{fig:flatnormconv}.
Instead, we use a polyhedral approximation to $T$ which guarantees that the mass increases by at most $\delta$ (i.e., $\mass(P_\delta) < \mass(T) + \delta$, rather than the simplicial deformation theorem bound $\mass(P_\delta) < C_1\mass(T) + C_2\mass(\boundary T)$ with the constants bounded away from 1).

The next step is to take an optimal (possibly nonintegral) decomposition of $T$ and approximate it with polyhedral chains (see~\cref{fig:approxoverview}).
That is, approximate the decomposition $T = X + \boundary S$ with polyhedral $X_\delta$ and $S_\delta$.
If these approximations naturally form a decomposition (not necessarily optimal) of $P_\delta$ (i.e., $P_\delta = X_\delta + \boundary S_\delta$), then we would have $\F_{K_\delta}(P_\delta) \leq \mass(X_\delta) + \mass(S_\delta) < \F(T) + 2\delta$ for any complex $K_\delta$ containing $P_\delta$, $X_\delta$, and $S_\delta$.
This result of course implies \cref{eq:fteqgoal}.
%the desired result $\lim_{n\rightarrow \infty} \F_{K_\delta}(P_\delta) = \F(T)$.
% from which the desired limit result follows.

\begin{figure}[h!]
\medskip
\centering
%\begin{tikzpicture}[scale=1,every node/.style={scale=1}]
%  \matrix (m) [matrix of math nodes,row sep=3em,column sep=3em,minimum width=2em]
%  {
%    T & P_\delta\\ % Add & Y_\delta + R_\delta to add simplicial decomposition of P
%    X+\boundary S & X_\delta + \boundary S_\delta\\
%  };
%  \path[-stealth]
%   (m-1-1) edge[-,draw opacity=0] node[rotate=90] {$=$} (m-2-1)
%            edge (m-1-2)
%   % Uncomment to fill in relationship between P and its simplicial decomposition (if added above)
%   %(m-1-2) edge[-,draw opacity=0] node {$=$} (m-1-3)
%   (m-2-1.east|-m-2-2) edge (m-2-2);
%\end{tikzpicture}
\begin{tikzpicture}[scale=1,every node/.style={scale=1}]
  \matrix (m) [matrix of math nodes,row sep=3em,column sep=7em,minimum width=2em]
  {
    T & P_\delta \\%&Y_\delta + \boundary R_\delta\\
    X+\boundary S & X_\delta + \boundary S_\delta\\
  };
  \path[-stealth]
   (m-1-1) edge[-,draw opacity=0] node[rotate=90] {$=$} node[left=1em,text width=6.5em,align=center] {optimal\\flat norm\\decomposition}(m-2-1)
            edge node[above=0.5em,text width=7em,align=center] {Polyhedral\\approximation}(m-1-2)
   %(m-1-2) edge[-,draw opacity=0] node {$=$} node[above=0.5em,text width=7em,align=center] {Optimal simplicial\\flat norm\\decomposition} (m-1-3)
   (m-2-1.east|-m-2-2) edge node [above=1em] {$X \Rightarrow X_\delta$}
            node [above] {$S \Rightarrow S_\delta$} node [below,text width=7em,align=center] {Polyhedral\\approximation} (m-2-2);
\end{tikzpicture}
\label{fig:approxoverview}
\caption[Overview of useful approximations and decompositions]{Various approximations and decompositions used in our results.}
\end{figure}
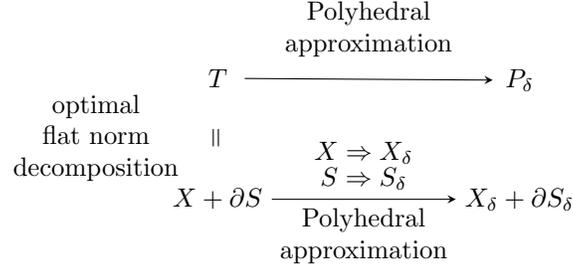

However (as shown in \cref{fig:polyhedralization}), we need not have $P_\delta = X_\delta + \boundary S_\delta$.
Since we obtained these quantities by polyhedral approximation, it turns out that the extent to which this equation is violated is small (in the continuous flat norm).
That is, we have 
\begin{align}
\label{eq:icgoalrestate}
P_\delta = X_\delta + \boundary S_\delta + (P_\delta - T) + (\boundary S - \boundary S_\delta) + (X - X_\delta).
\end{align}

While \cref{eq:icgoalrestate} can be viewed as a decomposition of $P_\delta$, the added error terms mean it may not be a chain on a simplicial complex.
Hence it cannot be used directly to bound the simplicial flat norm of $P_\delta$.

If we use the simplicial deformation theorem to push the error terms to some complex $K_\delta$ while preserving a pushed version of \cref{eq:icgoalrestate}, we can obtain a candidate simplicial decomposition of $P_\delta$.
In order to use this candidate to bound $\F_{K_\delta}$, we must know that this pushing step did not make the small error terms too large.
Unfortunately, the bounds on mass specified by the simplicial deformation theorem rely on simplicial regularity, so a sufficiently skinny simplex could mean the error terms become large.
If the simplicial irregularity in $K_\delta$ gets worse as $\delta \downarrow 0$, we will not be able to show \cref{eq:fteqgoal}.

Since we know exactly which currents we wish to push, our solution is to pick $K_\delta$ with these in mind: make sure the complex is as regular as possible overall (independently of $\delta$) with any irregularities (which may be required to embed $P_\delta$, $X_\delta$, and $S_\delta$) isolated in subcomplexes of small measure.
By making the irregular portions small enough (so they contain a negligible portion of the error terms, even considering the possible magnification from pushing), we establish a deformation theorem variant (\cref{thm:2dboundedsdt}) with constant mass expansion bounds, assuming a triangulation result that lets us isolate the irregularities as described (Shewchuk's Terminator algorithm\cite{Sh2002} provides this result in $\R^2$).
The pushed version of \cref{eq:icgoalrestate} allows us to prove $\F_{K_\delta}(P_\delta) \leq \F(T) + O(\delta)$, from which \cref{eq:fteqgoal} and \cref{thm:icmainresult} follow.

\begin{figure}[hb]
\medskip
\centering
\begin{tikzpicture}[scale=0.5]
	\tikzstyle{point}=[circle,thick,draw=black,fill=black,inner sep=0pt,minimum width=4pt,minimum height=4pt]
	\tikzstyle{current}=[thick]

	\node (t1)[point,label={a}] at (0,0) {};
	\node[right = 3cm] (t2)[point] {};
	\node[above right = 2cm and 1.5cm of t2] (t3)[point] {};
	\node[below right = 1.5cm and 1.5cm of t3] (t4)[point] {};
	\node[above right = 1cm and 1cm of t4] (t5)[point] {};
	\node[below right = 1cm and 1cm of t5] (t6)[point] {};

	\node[above right=0.1cm and 3.5cm of t1] (u2)[point] {};
	\node[above right=0.7cm and 0.7cm of u2] (u3a)[point] {};
	\node[right=0.6cm of u3a] (u3)[point] {};
	\node[below right=0.2cm and 0.5cm of u3] (u4)[point] {};
	\node[right=2.7cm of u4] (u5)[point] {};

	\node[below left=1cm and 1cm of t3] (s2)[point] {};
	\node[right=2.2cm of s2] (s3)[point] {};
	\node[below left=0.7cm and 0.65cm of t5] (s4)[point] {};
	\node[right=1.3cm of s4][point] (s5) {};

    \fill[opacity=0.5,color=red]  (t3) -- (s2) -- (s3) -- cycle;
     %\draw[current,color=red,dashed] (p.center) -- node[label={[label distance=1cm]176:$T$}] {} (q.center);
	\draw[current] (t1) -- node[label={270:$P_\delta$}] {} (t2) -- (t3) -- (t4) -- (t5) -- (t6);
	\draw[current,color=blue,dashed] (t1) -- node[label={$X_\delta$}] {} (u2) -- (u3a) -- (u3) -- (u4) -- (u5);
    \draw[thick,color=red,fill=red,fill opacity=0.4,text opacity=1]  (t3.center) -- node[label={$S_\delta$}] {} (s2.center) -- (s3.center) -- cycle;
    \draw[thick,color=red,fill=red,fill opacity=0.4,text opacity=1]  (t5.center) -- (s4.center) -- (s5.center) -- cycle;
	%\node (b)[point,label={b}] at (5,0) {};
\end{tikzpicture}
\caption[Polyhedral approximation of \cref{fig:icflatnorm}]{A possible polyhedral approximation of the decomposition shown in \cref{fig:icflatnorm}.
Note that $P_\delta \neq X_\delta + \boundary S_\delta$.}
\label{fig:polyhedralization}
\end{figure}
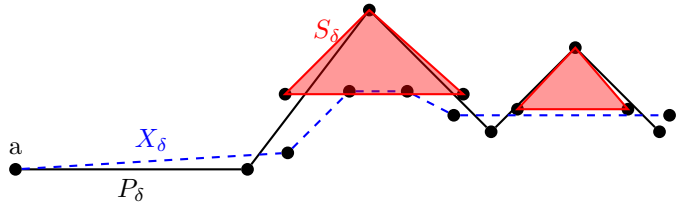

\section{Preliminaries}
Our goal is to investigate conditions under which the flat norm decomposition of an integral current can be taken to be integral as well.
The corresponding statement for normal currents is true and useful in our development.

\begin{lemma}
\label{lem:normaldecomp}
If $T$ is a normal $m$-current and $X$ and $S$ are $m$- and $(m+1)$-currents such that $T = X + \boundary S$ and $\F(T) = \mass(X) + \mass(S)$ (i.e., $T = X + \boundary S$ is a flat norm decomposition of $T$), then $X$ and $S$ are normal currents.
\end{lemma}
\begin{proof}
By the definition of normal current, we have $\mass(T) + \mass(\boundary T) < \infty$.
Thus
\[
\mass(X) + \mass(S) = \F(T) \leq \mass(T) < \infty
\]
so $\mass(X) < \infty$ and $\mass(S) < \infty$.
Since $T = X + \boundary S$, we obtain
\[
\mass(\boundary X) = \mass(\boundary\del{X + \boundary S}) = \mass(\boundary T) < \infty.
\]
Lastly,
\[
\mass(\boundary S) \leq \mass(\boundary S - T) + \mass(T) = \mass(-X) + \mass(T)  < \infty.
\]
The currents $X$ and $S$ have compact support by the definition of the flat norm.
Thus $X$ and $S$ are normal by definition.
\end{proof}

Convergence in the flat norm is linear and commutes with the boundary operator as the following easy lemma shows.

\begin{lemma}
\label{lem:convprops}
Suppose that $T_n$ and $U_n$ are $m$-currents for $n = 1, 2, \dots$, and $T_n \rightarrow T$ and $U_n \rightarrow U$ in flat norm (i.e., $\F(T_n - T) \rightarrow 0$, $\F(U_n - U) \rightarrow 0$) for some $m$-currents $T$ and $U$.
%(that is, $\lim_{n\rightarrow \infty} \langle T_n - T, \phi\rangle = 0$ for every differential $m$-form $\phi$ with compact support) and $U_n \rightarrow U$ for some $m$-currents $T$ and $U$.
Then the following properties hold:
\begin{enumerate*}
\item[(a)] $\alpha T_n + \beta U_n \rightarrow \alpha T + \beta U$ for any constants $\alpha, \beta \in \R$; and
\item[(b)] $\boundary T_n \rightarrow \boundary T$.
\end{enumerate*}
\end{lemma}
\begin{proof}
We apply properties of norms to obtain
\begin{align*}
\F((\alpha T_n + \beta U_n) - (\alpha T + \beta U)) &\leq \F(\alpha T_n - \alpha T) + \F(\beta U_n - \beta U)\\
&= \abs{\alpha}\F(T_n - T) + \abs{\beta}\F(U_n - U).
\end{align*}
Letting $n \rightarrow \infty$ yields the linearity result.
Now let $X_n$ and $S_n$ be $m$- and $(m+1)$-currents such that $X_n + \boundary S_n$ is a flat norm decomposition of $T_n - T$ for $n = 1, 2, \dots$, observing that
\begin{align*}
\F(\boundary T_n - \boundary T) &= \F(\boundary (X_n + \boundary S_n))
= \F(\boundary X_n)
\leq \mass(X_n)
\leq \F(T_n - T).
\end{align*}
The boundary result follows in the limit.
\end{proof}

In the case of the simplicial flat norm, an input integral chain is guaranteed an integral chain decomposition whenever the simplicial complex is totally unimodular\cite{IbKrVi2013}.
This occurs when the complex is free of relative torsion, which is the case for any $(d+1)$-complex in $\R^{d+1}$ or when triangulating a compact, orientable $(d+1)$-dimensional manifold\cite{DeHiKr2011}.

\begin{theorem}[Simplicial flat norm integral decomposition\cite{IbKrVi2013}]
\label{thm:simpint}
If $K$ is a simplicial $(d+1)$-complex embedded in $\R^{d+1}$, then for any integral $d$-chain $P$ on $K$, the optimal simplicial flat norm value for $P$ is attained by an integral decomposition.
\end{theorem}

We state the simplicial deformation theorem and sketch a portion of its proof.
We will later modify it to obtain a multiple current deformation theorem that preserves linearity (\cref{thm:multisdt}).

\begin{theorem}[Simplicial deformation theorem\cite{IbKrVi2013}]
\label{thm:sdt}
Suppose $K$ is a $p$-dimensional simplicial complex in $\R^q$ and $T$ is a normal $d$-current supported\footnote{Intuitively, the support of an integral $d$-current is the $d$-dimensional surface it represents.
More precisely, a current $T$ is supported in a set $C$ if $T(\phi) = 0$ for every differential form $\phi$ compactly supported in the complement of $C$.}
on the underlying space of $K$.
There exists a simplicial $d$-current $P$ supported on the $d$-skeleton of $K$ with boundary supported on the $(d-1)$-skeleton (i.e., a simplicial $d$-chain) such that $T-P=Q+\boundary R$ and there exists a constant $\upvartheta_K$ (depending only on simplicial regularity in $K$) such that the following controls on mass hold:
\begin{align}
\mass(P) &\leq (4\upvartheta_K)^{p-d} \mass(T) + \Delta (4\upvartheta_K)^{p-d+1} \mass(\boundary T),\\
\mass(\boundary P) &\leq (4\upvartheta_K)^{p-d+1}\mass(\boundary T),\\
\mass(Q) &\leq \Delta (4\upvartheta_K)^{p-d}(1+4\upvartheta_K)\mass(\boundary T),\\
\mass(R) &\leq \Delta(4\upvartheta_K)^{p-d}\mass(T), ~\mbox{ and }\\
\F(T,P) &\leq \Delta (4\upvartheta_K)^{p-d}(\mass(T)+(1+4\upvartheta_K)\mass(\boundary T)),
\end{align}
where $\Delta$ is the largest diameter of a simplex in $K$.
The regularity constant $\upvartheta_K$ is given by
\begin{equation}
\label{eq:regularity}
\upvartheta_K = \sup_{\sigma \in K}\frac{\diam(\sigma)\perimeter(\sigma)}{B_\sigma} + 2 \sup_{\sigma \in K}\frac{\diam(\sigma)}{\inrad(\sigma)}
\end{equation}
where for each $l$-simplex $\sigma$, $\perimeter(\sigma)$ is the $(l-1)$-volume of $\boundary \sigma$ and $B_\sigma$ is the $l$-volume of a ball with radius $\inrad(\sigma)/2$ in $\R^l$.
\end{theorem}
\begin{proof}[Proof highlights]
The simplicial current $P$ is obtained by pushing $T$ and its boundary to the $d$- and $(d-1)$-dimension skeletons of $K$ respectively.
This pushing is done one dimension at a time; that is, $T$ is pushed from the $p$-skeleton (i.e., the full complex $K$) to the $(p-1)$-skeleton, then to the $(p-2)$-skeleton, and so on until the $d$-skeleton.
Pushing the current from the $\ell$-skeleton to the $(\ell-1)$-skeleton is done by picking a projection center in each $\ell$-simplex $\sigma$ and projecting the current in $\sigma$ outwards to $\boundary \sigma$ via straight-line projection.

A crucial step in the proof is to find a projection center that bounds the expansion of $T$ and $\boundary T$.
In particular, this is done by proving that over all possible centers, the average expansion is bounded and then showing that individual centers exist with bounded expansion.
We call out this particular step because we modify it to obtain \cref{thm:multisdt}.

When projecting onto the skeleton of each simplex $\sigma$, we have\cite[Lemma 5.9]{IbKrVi2013}
\begin{equation}
\label{eq:avgbound}
 \frac{1}{B_\sigma}\int_{{\cal B}_\sigma} \int_{\sigma} J_d \phi(\vx,\va) \, {\rm
  d}\norm{T}(\vx) \, {\rm d} {\cal L}^{\ell}(\va) 
%= \int_{\sigma}
%  \int_{{\cal B}_\sigma} J_d \phi(\vx,\va) \, {\rm d}{\cal
%  L}^{\ell}(\va) \, {\rm d} \norm{T}(\vx) \,
\, \leq \,
  \upvartheta_{\sigma} \mass(T|_{\sigma}),
\end{equation}
where ${\cal B}_\sigma$ is the set of possible centers in $\sigma$, $B_\sigma$ is its $\ell$-volume, and $\upvartheta_{\sigma}$ is a regularity constant for $\sigma$ given by
\[
\upvartheta_\sigma = \frac{\diam(\sigma)\perimeter(\sigma)}{B_\sigma} + 2\frac{\diam(\sigma)}{\inrad(\sigma)}.
\]

\cref{eq:avgbound} shows that in each projection step the mass of $T$ expands by a factor of at most $\upvartheta_{K}$ averaged over all possible choices of centers.
As the average expansion over all centers is bounded by $\upvartheta_K$, we observe that at most $\frac{1}{4}$ of the possible centers can expand the mass of $T$ by a factor of $4\upvartheta_{K}$ or more.
Similarly, at most $\frac{1}{4}$ of the centers can expand $\boundary T$ by a factor of $4\upvartheta_{K}$ or more.
Therefore, at least $\frac{1}{2}$ of the possible centers bound the expansion of both $T$ and $\boundary T$ by at most a factor of $4\upvartheta_K$.
Choosing a center from this set for each simplex yields the bounds required in the theorem.
\end{proof}

The following theorem allows normal (or integral) currents to be approximated by polyhedral chains, which are not necessarily simplicial chains contained in an a priori complex.
Note in particular that the mass bounds can be made arbitrarily tight by the choice of $\rho$, which is in contrast with the larger bounds of the deformation theorems.

\begin{theorem}[Polyhedral approximation of currents\cite{Federer1969}, 4.2.21, 4.2.24]
\label{thm:polyapprox}
If $\rho > 0$ and $T$ is a normal $m$-current in $\R^n$ supported in the interior of a compact subset $K$ of $\R^n$, then there exists a polyhedral chain $P$ with
\begin{subequations}
\begin{align}
\F(P-T) & \leq \rho,\\
\label{eq:Pbound}\mass(P) & < \mass(T) + \rho,~\mbox{ and }\\
\label{eq:BPbound}\mass(\boundary P) & < \mass(\boundary T) + \rho.
\end{align}
\end{subequations}
If $T$ is integral, then $P$ can be taken to be integral as well.
\end{theorem}
\begin{proof}
This is a slight modification of Federer's theorems which do not state \cref{eq:Pbound,eq:BPbound} separately, but rather present a combined bound $\mass(P) + \mass(\boundary P) \leq \mass(T) + \mass(\boundary T) + \rho$.
We show only the derivation of the separated bounds.

In the normal current case~\cite[4.2.24]{Federer1969}, these bounds follow from Federer's proof.
In particular, we have currents $P_1$, $P_2$, and $Y$ such that $P = P_1 + Y$ and the following bounds hold:
\begin{subequations}
\label{eq:Fbounds}
\begin{align}
\label{eq:FB1}\mass(P_1) < \mass(T) + \rho/4,\\
\label{eq:FB2}\mass(P_2) < \mass(\boundary T) + \rho/4,\\
\label{eq:FB3}\mbox{and }~\mass(P_2 - \boundary P_1 - \boundary Y) + \mass(Y) < \rho/2.
\end{align}
\end{subequations}
The bounds in \cref{eq:Pbound,eq:BPbound} follow from the triangle inequality and \cref{eq:FB1,eq:FB2,eq:FB3}:
\begin{align*}
\mass(P) &\leq \mass(P_1) + \mass(Y)\\
&< \mass(T) + \rho/4 + \rho/2,\\
\mass(\boundary P) &= \mass(\boundary P_1 + \boundary Y)\\
&\leq \mass(P_2 - \boundary P_1 - \boundary Y) + \mass(P_2)\\
&< \rho/2 + \mass(\boundary T) + \rho/4.
\end{align*}

In the integral current case~\cite[4.2.21]{Federer1969}, Federer applies the approximation theorem 4.2.20 to obtain $P$ close to the pushforward of $T$ under a Lipschitz diffeomorphism $f$.
That is, for any fixed $\epsilon > 0$, there exist $P$ and $f$ such that
\begin{subequations}
\begin{align}
\label{eq:Fpushbound}\mass(P - f_\# T) + \mass(\boundary P - \boundary f_\# T) \leq \epsilon,\\
\label{eq:Fpushlip1}\lip(f) \leq 1 + \epsilon,\\
\label{eq:Fpushlip2}\mbox{ and }~\lip(f^{-1}) \leq 1 + \epsilon.
\end{align}
From \cref{eq:Fpushbound,eq:Fpushlip1,eq:Fpushlip2}, we obtain mass bounds on $P$ and $\boundary P$:
\end{subequations}
\begin{subequations}
\begin{align}
\mass(P) &\leq \mass(f_\# T) + \epsilon\\
&\leq (1+\epsilon)^m \mass(T) + \epsilon,\\
\mass(\boundary P) &\leq \mass(\boundary f_\# T) + \epsilon\\
&\leq (1+\epsilon)^{m-1} \mass(\boundary T) + \epsilon.
%\mass(P) &\leq \mass(f_\# T) + \epsilon
\end{align}
\end{subequations}
%\begin{align}
%\mass(P) &\leq \mass(f_\# T) + \epsilon\\
%& (1+\epsilon)^m \mass(T) + \epsilon
%\end{align}
The bounds in \cref{eq:Pbound,eq:BPbound} follow by choosing $\epsilon$ small enough.
\end{proof}
\section{Results} \label{sec:results}

We modify the simplicial deformation theorem to allow multiple currents to be deformed simultaneously by projecting from the same centers.
As opposed to using \cref{thm:sdt} separately on each current (where the centers of projection need not be the same), this approach yields a linearity result: deformations of linear combinations are linear combinations of deformations.
Pushing multiple currents at the same time comes at the cost of looser bounds on the deformation (linear in the number of currents), although more careful analysis tightens these bounds by approximately a (constant) factor of $2$ (compare \cref{thm:sdt,cor:multisdtsdt} in the single current case).

\begin{theorem}
\label{thm:multisdt}
Suppose $\epsilon > 0$ and we have the hypotheses of \cref{thm:sdt} except that there are now $m$ $d$-currents $T_1, T_2, \dots, T_m$ and $n$ $(d+1)$-currents $S_1, S_2, \dots, S_n$ to push on to the complex to yield the corresponding simplicial chains $P_i$ and $O_j$.
There is a series of projection centers (as in the proof of \cref{thm:sdt} and depending on $\epsilon$, $K$, the $T_i$ and $S_j$) which can be used with every current $T_i$ and $S_j$ to obtain the bounds:

\begin{align*}
\mass(P_i) &\leq ((2m+2n +\epsilon)\upvartheta_K)^{p-d} \mass(T_i) + \Delta ((2m+2n+\epsilon)\upvartheta_K)^{p-d+1} \mass(\boundary T_i),\\
\mass(\boundary P_i) &\leq ((2m+2n+\epsilon)\upvartheta_K)^{p-d+1}\mass(\boundary T_i),\\
\F(T_i,P_i) &\leq \Delta ((2m+2n+\epsilon)\upvartheta_K)^{p-d}(\mass(T_i)+(1+(2m+2n+\epsilon)\upvartheta_K)\mass(\boundary T_i)),\\
\mass(O_j) &\leq ((2m+2n+\epsilon)\upvartheta_K)^{p-d-1} \mass(S_j) + \Delta ((2m+2n+\epsilon)\upvartheta_K)^{p-d} \mass(\boundary S_j),\\
\mass(\boundary O_j) &\leq ((2m+2n+\epsilon)\upvartheta_K)^{p-d}\mass(\boundary S_j), ~\mbox{ and } \\
\F(S_j,O_j) &\leq \Delta ((2m+2n+\epsilon)\upvartheta_K)^{p-d-1}(\mass(S_j)+(1+(2m+2n+\epsilon)\upvartheta_K)\mass(\boundary S_j)).
\end{align*}
Moreover, if we let $\pi_K$ denote the projection map that uses these centers to push $(d-1)-$, $d-$, and $(d+1)$-currents to chains on the complex, then we have that:
\begin{itemize}
\item $\pi_K$ commutes with the boundary operator (i.e., $\pi_K(\boundary A) = \boundary \pi_K(A)$ where $A$ is any $d$- or $(d+1)$-current)
\item $\pi_K$ is linear on the currents $T_i$, $\boundary T_i$, $S_j$ and $\boundary S_j$.
That is, for any scalars $a_i$ and $b_j$,
\begin{align*}
\pi_K\left(\sum_{i=1}^m a_i \boundary T_i\right) &= \sum_{i=1}^m a_i \pi_K(\boundary T_i),\\
%&= \sum_{i=1}^m a_i \boundary \pi_K(T_i)\\
\pi_K\left(\sum_{i=1}^m a_i T_i + \sum_{j=1}^n b_j \boundary S_j\right) &= \sum_{i=1}^m a_i \pi_K(T_i) + \sum_{j=1}^n b_j \boundary (\pi_K(S_j)), ~\mbox{ and }\\
\pi_K\left(\sum_{j=1}^n b_j S_j \right) &= \sum_{j=1}^n b_j \pi_K(S_j).
\end{align*}
\end{itemize}
\end{theorem}
\begin{proof}
We must show that there are centers in the set of feasible centers ${\cal B}_\sigma$ (see the proof sketch of \cref{thm:sdt}) which simultaneously achieve the various bounds on the $2(m+n)$ relevant currents: $T_1, \dots, T_m$, $\boundary T_1, \dots \boundary T_m$, $S_1, \dots S_n$, and $\boundary S_1, \dots, \boundary S_n$.

We consider the case of projecting currents from the $\ell$-skeleton to the $(\ell-1)$-skeleton in the $\ell$-simplex $\sigma$.
As in the proof of \cref{thm:sdt}, we again use the average bound in \cref{eq:avgbound}.
For each $k \in \Z^+$ and $i = 1, 2, \dots, m$, let 
\[
H_{T_i,k} = \left\{\va \in {\cal B}_\sigma \,\middle|\, \int_\sigma J_d \phi(\vx,\va) \, {\rm d}\norm{T_i}(\vx) > \left(2m+2n+\frac{1}{k}\right)\upvartheta_\sigma\mass(T_i)\right\}.
\] 
Then, using the same average-based argument as in \cref{thm:sdt}, we have that $\mathcal{H}^\ell(H_{T_i,k})/\mathcal{H}^\ell({\cal B}_\sigma) < \frac{1}{2m+2n}$ (i.e., the size of the set of poorly behaved centers with respect to each $T_i$ is a small fraction of the set ${\cal B}_\sigma$ of possible centers).
We similarly define $H_{\boundary T_i,k}$, $H_{S_j,k}$, and $H_{\boundary S_j, k}$ and obtain the same bound of $\frac{1}{2m+2m}$ on the bad centers.
For each $k \in \Z^+$, we are interested in the set of centers which are simultaneously good centers for all currents involved (i.e., points in ${\cal B}_\sigma$ but not any of the $H_{\cdot, k}$ sets).
Call this set $G_k$ and observe that it has positive measure:
\begin{align*}
{\cal H}^\ell(G_k) &={\cal H}^\ell\left({\cal B}_\sigma \backslash\left(\bigcup_{i=1}^m H_{T_i, k} \cup \bigcup_{i=1}^m H_{\boundary T_i, k}\cup \bigcup_{i=1}^n H_{S_i,k} \cup \bigcup_{i=1}^n H_{\boundary S_i,k}\right)\right)\\
&\geq {\cal H}^\ell({\cal B}_\sigma) - \sum_{i=1}^m{\cal H}^\ell(H_{T_i, k}) - \sum_{i=1}^m{\cal H}^\ell(H_{\boundary T_i, k}) - \sum_{j=1}^n{\cal H}^\ell(H_{S_j,k}) - \sum_{j=1}^n{\cal H}^\ell(H_{\boundary S_j,k})\\
&> {\cal H}^\ell({\cal B}_\sigma)\left(1 - \frac{m}{2m+2n} - \frac{m}{2m+2n} - \frac{n}{2m+2n} - \frac{n}{2m+2n}\right)\\
&=0.
\end{align*}
Thus for any $k > \frac{1}{\epsilon}$ we have that $G_k$ is a nonempty set of possible projection centers which simultaneously attain an expansion bound of at most $(2m+2n + \epsilon)\upvartheta_\sigma$ for all the pertinent currents.

The projection operator is clearly linear and commutes with the boundary operator as a consequence of properties\cite[4.1.6]{Federer1969} of the differential forms to which currents are dual.
\end{proof}

\begin{corollary}
\label[corollary]{cor:multisdtsdt}
The bounds in \cref{thm:sdt} can all be tightened by replacing $4\upvartheta_K$ with $(2+\epsilon) \upvartheta_K$ for $\epsilon > 0$.
\end{corollary}
\begin{proof}
Simply taken $m = 1$ and $n = 0$ in \cref{thm:multisdt}.
\end{proof}

For a $2$-complex $K$, the minimum angle over all triangles in the complex is easier to work with, and can be used as a proxy for our simplicial regularity constant as \cref{lem:2danglereg} indicates.

\begin{lemma}
\label{lem:2danglereg}
A lower bound on the minimum angle of all triangles in a $2$-complex implies an upper bound on the simplicial regularity constant.
That is, given a $2$-complex $K$ with minimum angle at least $\theta$, we have $\upvartheta_K \leq C_\theta$ for some constant $C_\theta$.
\end{lemma}
\begin{proof}
The simplicial regularity constant $\upvartheta_K$ used for \cref{thm:sdt,thm:multisdt} in the case of triangles is given by
\[
\upvartheta_K = \frac{4}{\pi}\sup_{\sigma \in K}\frac{\diam(\sigma)\perimeter(\sigma)}{\inrad(\sigma)^2} + 2 \sup_{\sigma \in K}\frac{\diam(\sigma)}{\inrad(\sigma)}.
\]
We observe that bounding $\diam(\sigma)/\inrad(\sigma)$ and $\perimeter(\sigma)/\inrad(\sigma)$ for all triangles $\sigma \in K$ yields a bound for $\upvartheta_K$.
Suppose $\sigma$ has side lengths $a \geq b \geq c$ and angle $\gamma$ opposite $c$.
Using the law of cotangents, we obtain
\[
\frac{\diam(\sigma)}{\inrad(\sigma)} = \frac{a\cot(\gamma/2)}{(a+b)/2-c/2} \leq \frac{a\cot(\gamma/2)}{(a+b)/2-b/2} = 2\cot(\gamma/2) \leq 2\cot(\theta/2).
\]
The bound for $\perimeter(\sigma)/\inrad(\sigma)$ follows easily from this observation:
\[
\frac{\perimeter(\sigma)}{\inrad(\sigma)} \leq \frac{3\diam(\sigma)}{\inrad(\sigma)} < 6\cot(\theta/2).
\]
Thus we can take $C_\theta = \frac{48}{\pi}\cot(\theta/2)^2 + 4\cot(\theta/2)$.
\end{proof}

\medskip
Our result relies on the ability to localize irregularities via subdivision, focusing on localization rather than removal because the latter is not possible.
For example, any subdivision of a 2-complex with a very small input angle will have an angle that is at least as small.
With that in mind, we require that subdivisions be possible which push the irregularities into the corners.
That is, the irregularity should be bounded by a constant (independent of the complex) away from the skeleton of the original complex and a complex-dependent constant (reflecting the necessity of some bad simplices) near the skeleton.
\cref{con:boundreg} formalizes this requirement and \cref{thm:2dboundreg} notes some cases where it holds.
We present our main theorem in such a way that proving \cref{con:boundreg} more generally will automatically extend our results.

\begin{conjecture}
\label{con:boundreg}
For any $p$-dimensional simplicial complex $K$ in $\R^q$ and $\epsilon > 0$, it is possible to subdivide $K$ so that all simplices are of bounded ``badness'' (with bound independent of $K$ or $\epsilon$) except possibly for simplices in a region of $p$-dimensional volume less than $\epsilon$ near the $(p-1)$-skeleton; even these simplices have bounded badness (dependent on $K$ but not $\epsilon$).
More precisely, there exists a subdivision $M_\epsilon$ of $K$ and a subcomplex $M_\epsilon'$ of $M_\epsilon$ (with simplicial regularity constants $\upvartheta_{M_\epsilon}$ and $\upvartheta_{M_\epsilon'}$) such that:
\begin{enumerate}
\item $M_\epsilon \backslash M_\epsilon' \subseteq \{\vx \in \R^q \,\mid\, \|\vx-\vy\| < \epsilon \text{ for some $\vy$ in the $(p-1)$-skeleton of $K$}\}$,
\item $\upvartheta_{M_\epsilon} \leq \alpha_K$ for some constant $\alpha_K$, and
\item $\upvartheta_{M_\epsilon'} \leq \beta$ for some fixed constant $\beta$.
\end{enumerate}
In particular, $\alpha_K$ does not depend on $\epsilon$ and $\beta$ does not depend on $K$ or $\epsilon$.
The simplicial regularity constants are defined as in \cref{eq:regularity}.
\end{conjecture}

While we will only prove this conjecture for certain $1$- and $2$-dimensional simplicial complexes, we are not aware of any counterexamples that would constrain the conjecture in higher dimensions.

\begin{theorem}
\label{thm:2dboundreg}
\cref{con:boundreg} holds for:
\begin{itemize}
\item $q \geq p = 1$, ~\mbox{ and }
\item $p = q = 2$.
\end{itemize}
\end{theorem}
\begin{proof}
The $p = 1$ case is trivial as all 1-simplices have the same regularity so we have $\upvartheta_K = 8$ and can take $M_\epsilon = M_\epsilon' = K$.

For the $p = q = 2$ case, we proceed in two steps.
First we will superimpose a square grid on $K$ (orienting it to bound the minimum angle created between its edges and those of $K$), creating a cell complex which is a refinement of $K$.
Next we use Shewchuk's Terminator algorithm\cite{Sh2002} to further refine the cell complex back into a simplicial complex with bounds on the minimum angle and, crucially, restrictions on where these small angles can be so that we can obtain regularity bounds.

By superimposing a fine enough square grid, we can force the small angles (whether already present in the complex or newly created) to occur only in a small measure subset of the complex.
Pick $\delta > 0$ small enough so that the set
\[
\{\vx \in \R^2 \,\mid\, \vy\text{ lies on the $1$-skeleton of $K$, } \|\vx-\vy\| < 3\delta\}
\]
 has measure less than $\epsilon$.
Let $G$ be a finite square grid in $\R^2$ whose cells each have diameter $\delta$ such that $G$ covers the underlying space of $K$ in any rotation.
Note that there are only two directions present in $G$ so if we bound all possible angles created between these directions and the edges of $K$, we can bound the minimum new angle created by superimposing $G$.

Let $\vw \in \R^2$ be a fixed unit vector and define
\[
E = \left\{\phi, \phi + \frac{\pi}{2} \,\mid\, \phi \text{ is the angle between $\vu-\vv$ and $\vw$ for some edge $(\vu,\vv) \in K$}\right\}.
\]
Further let $E^\theta = \{\psi \in [0, 2\pi) \mid |\phi-\psi| < \theta\text{ for some }\phi \in E\}$.
This is the set of angles to avoid when rotating $G$ in order to guarantee all created angles will be $\theta$ or larger.

Denote by $\eta < \infty$ the cardinality of $E$ and note that $[0, 2\pi) \backslash E^\frac{\pi}{2\eta}$ has positive measure, so there exist rotations of the square grid that create no new angles smaller than $\frac{\pi}{2\eta}$.

After superimposing a suitably rotated version of $G$, we obtain a new cellular complex which is a refinement of $K$.
This is a planar straight line graph which can be used as input to Shewchuk's Terminator algorithm\cite{Sh2002}, which refines it into a simplicial complex $M_\epsilon$ with guarantees on the minimum angle bound of the resulting complex and where the small angles can occur.

In particular, if $\theta$ is the minimum angle in the cellular complex (either present originally or created by the square grid superposition), then the minimum angle of $M_\epsilon$ is at least $\arcsin((\sqrt{3}/2)\sin(\theta/2))$.
Furthermore, no angles less than $30^\circ$ are created by the algorithm except in the vicinity of angles less than $60^\circ$.

Specifically, the algorithm proceeds by iteratively splitting segments and triangles and has a concept of a subsegment cluster which is a collection of nearby subsegments such that splitting one triggers splits in all of the others.
A vertex is said to encroach on a subsegment if it is contained within the circle using the subsegment as a diameter.
Encroachment is used to identify candidates for further splitting.

Any newly created small angles must be part of a skinny triangle whose circumcenter encroaches upon a subsegment cluster bearing a small input angle.
As all such subsegment clusters must be contained within a distance of $2\delta$ of the 1-skeleton of $K$ by design, we have that all small angles in $M_\epsilon$ are within $3\delta$ of the 1-skeleton of $K$.

Let $M_\epsilon'$ be the subcomplex of $M_\epsilon$ containing all triangles not fully contained in the $3\delta$ tube.
Noting that all angles in $M_\epsilon'$ are at least $30^\circ$, we get by \cref{lem:2danglereg} that
\[
\upvartheta_{M_\epsilon'} \leq \frac{48}{\pi}\cot(15^\circ)^2 + 4\cot(15^\circ) = \frac{4(2+\sqrt{3})(24+12\sqrt{3}+\pi)}{\pi}. %\approx 227.74
\]
We may take $\beta$ to be this quantity, noting that it is independent of $\epsilon$ and $K$.
The minimum angle bound $\theta$ for $M_\epsilon$ and \cref{lem:2danglereg} give us a bound $\alpha_K$ for $\upvartheta_{M_\epsilon}$ (independent of $\epsilon$).
\end{proof}

The following theorem shows that the bounds in \cref{thm:multisdt} may be replaced with constants independent of the complex as well as the currents involved, if we subdivide the complex by means of \cref{con:boundreg} (the subdivision does depend on the currents and the complex, of course).

\begin{theorem}
\label{thm:2dboundedsdt}
Suppose we have integers $d < s \leq q$ and that \cref{con:boundreg} holds for the given $q$ and every $p$ such that $d-1 \leq p$ and $p \leq s$ (that is, suppose we can isolate the irregularities of any $p$-complex in $\R^q$ by suitable subdivision).
Given an $s$-dimensional simplicial complex $K$ in $\R^q$ and a set of $d$-currents $T_1, \dots, T_m$ and $(d+1)$-currents $S_1, \dots, S_n$ in the underlying space of $K$ with $d < s$, there exists a complex $K'$ which is a subdivision of $K$ such that we have all of the conclusions of \cref{thm:multisdt} (i.e., mass and flat norm bounds and linear projection of the $T_i$ and $S_j$ to $K'$) except the simplicial regularity constant $\upvartheta_{K'}$ in the various bounds can be replaced with a constant $L$ that does not depend on $K$.
\end{theorem}
\begin{proof}
In the simplicial deformation theorems, the current is projected step-by-step to lower dimensional skeletons.
For instance, a $d$-current is projected from the initial $p$-complex to the $(p-1)$-skeleton, then to the $(p-2)$-skeleton, and  eventually down to the $d$-skeleton with one more step to push the current's boundary to the $(d-1)$-skeleton. Each projection is done by picking a center in each simplex and using it to project outward to the boundary of the simplex.
The simplicial regularity constant is used to bound the expansion of mass at each projection step and is defined by \cref{eq:regularity}, a bound on the regularity of all simplices in the complex.

However, this is a bit stronger than required as the projection is a local operation and the bound at each step depends only on the simplicial regularity of the simplex in question.
In addition, there is no reason in principle that we cannot subdivide the complex in between steps.
That is, after pushing to the $\ell$-skeleton, we can further subdivide the complex and then push to the newly refined $(\ell-1)$-skeleton.
In this case, the subdivision need not preserve the simplicial regularity of the $(\ell+1)$- or higher dimensional simplices, as all subsequent pushing steps will take place in lower dimensional simplices.
Moreover, for a given portion of the current, we can use the maximum of the simplicial regularity constants of the simplices it encounters while being pushed (rather than the maximum over {\em all} simplices in the complex).

For all $\epsilon > 0$ and nonnegative integers $k < p$, let $N_{k}^\epsilon$ denote the set of all points in the $(k+1)$-skeleton of $K$ with positive distance less than $\epsilon$ from the $k$-skeleton of $K$ (i.e., all points in the interior of the $(k+1)$-simplices of $K$ which are close to the $k$-skeleton).
Let $T \currentrestr N_{p-1}^\epsilon$ denote the restriction of the current $T$ to the set $N_{p-1}^\epsilon$ and note that since normal currents are representable by integration, the mass of $T \currentrestr N_{p-1}^\epsilon$ can be represented by the integral of a measure over $T \currentrestr N_{p-1}^\epsilon$.
As $\epsilon \downarrow 0$, the measure (and thus the mass of the current) goes to zero:
\begin{align}
\label{eq:limrestr}
\begin{split}
\lim_{\epsilon \downarrow 0} \mass(T_i \currentrestr N_{p-1}^\epsilon) = 0, & \quad
\lim_{\epsilon \downarrow 0} \mass(S_j \currentrestr N_{p-1}^\epsilon) = 0, \\
\lim_{\epsilon \downarrow 0} \mass(\boundary T_i \currentrestr N_{p-1}^\epsilon) = 0, &\quad
\lim_{\epsilon \downarrow 0} \mass(\boundary S_j \currentrestr N_{p-1}^\epsilon) = 0.
\end{split}
\end{align}
Let 
\begin{equation}
\label{eq:currentdelta}
\delta = \frac{\beta}{\alpha_K}\, \min_{1\leq i \leq m,1 \leq j \leq n}\{\mass(T_i), \mass(\boundary T_i), \mass(S_j), \mass(\boundary S_j)\}
%\delta = \frac{2^{1/(p-d+1)}-1}{2(m+n)\alpha_K}\min_{1\leq i \leq m,1 \leq j \leq n}\{\mass(T_i), \mass(\boundary T_i), \mass(S_j), \mass(\boundary S_j)\}
\end{equation}
 where $\alpha_K$ and $\beta$ are as in \cref{con:boundreg} and choose $\epsilon > 0$ to make each of the masses in \cref{eq:limrestr} less than $\delta$.
We can apply \cref{con:boundreg} with this $\epsilon$ to obtain a subdivision $M_\epsilon$ of $K$ and a subcomplex $M_\epsilon'$ such that the portion of each of our currents which lies in $M_\epsilon \backslash M_\epsilon'$ and is not already on the $(p-1)$-skeleton (so is not fixed by the first projection) has mass less than $\delta$.
This portion of each current increases in mass by a factor of at most $(2m+2n+\epsilon)\alpha_K$ when projecting to the $(p-1)$-skeleton (see proof of \cref{thm:multisdt}).
Letting $T_i'$ denote the result of projecting $T_i$ to the $(p-1)$-skeleton, we can bound its mass using \cref{eq:currentdelta}:
\begin{align*}
\mass(T_i') &\leq (2m+2n+\epsilon)\Big[\beta\mass(T_i\currentrestr M_\epsilon'\backslash \skel{p-1}(K)) + \alpha_K\mass(T_i \currentrestr M_\epsilon\backslash (M_\epsilon' \cup \skel{p-1}(K)))\Big] \\
&\quad+ \mass(T_i\currentrestr \skel{p-1}(K))\\
&\leq (2m+2n+\epsilon)(\beta\mass(T_i) + \alpha_K\delta)\\
% + \mass(T_i\currentrestr \skel{p-1}(K))\\
&\leq (2m+2n+\epsilon)(\beta\mass(T_i) + \beta\mass(T_i )) \\
&\leq (2m+2n+\epsilon)(2\beta)\mass(T_i).
\end{align*}
Similar inequalities hold for $S_j$, $\boundary T_i$, and $\boundary S_j$.
In the preceding argument, we have accomplished the goal of projecting all currents involved from the $p$-skeleton to the $(p-1)$-skeleton, and can now consider them as currents in the underlying space of the $(p-1)$-complex $\skel{p-1}(K)$.
We can apply this procedure iteratively (use \cref{con:boundreg} to localize the irregularities and then project) to push to the $(p-2)$- and lower-dimensional skeletons.

When we subdivide each $k$-skeleton using \cref{con:boundreg}, the higher dimensional simplices are not subdivided by default but this is easy to fix.
After a $k$-simplex is subdivided, add a point to the interior of every $(k+1)$-simplex of which it was a face, and connect the new point to every $k$-simplex on its boundary.
This step will likely generate highly irregular simplices, but since we have already pushed the currents down beyond their dimension, it is not an issue.

This argument continues in the same way as \cref{thm:sdt,thm:multisdt}, and proves our result with $L = 2\beta$.
\end{proof}

\noindent We now present our main result.

\begin{theorem}
\label{thm:icmainresult}
If $T$ is an integral $d$-current in $\R^{d+1}$ and \cref{thm:2dboundedsdt} holds for $d$, i.e., $d$- and $(d+1)$-currents can be well approximated by simplicial currents using subdivisions specified by \cref{con:boundreg}, then some flat norm minimizer for $T$ is an integral current.
That is, there is an integral $d$-current $X_I$ and integral $(d+1)$-current $S_I$ such that $\F(T) = \mass(X_I) + \mass(S_I)$ and $T = X_I + \boundary S_I$.
\end{theorem}
\begin{proof}
%TODO: Finite length
We let $X + \boundary S$ be an optimal flat norm decomposition of $T$.
That is, $X$ is a $d$-current and $S$ is a $(d+1)$-current such that
\begin{align}
\label{eq:Tdecomp}
T &= X + \boundary S & \mbox{ and \hspace*{1in}}  \F(T) &= \mass(X) + \mass(S).
\end{align}
We note by \cref{lem:normaldecomp} that $X$ and $S$ are normal currents.

As a general outline of the proof, for each $\delta > 0$, we will choose a particular simplicial complex $K_\delta$ on which we have $d$-chains $P_\delta$ and $X_\delta$  as well as $(d+1)$-chain $S_\delta$ respectively approximating $T$, $X$, and $S$ with error at most $\delta$.
We convert the (possibly nonintegral) optimal flat norm decomposition of $T$ into a candidate simplicial decomposition of $P_\delta$ in order to show (\cref{claim:FTeq}) the simplicial flat norm of $P_\delta$ converges to the flat norm of $T$ (this step does not yet show that the flat norm decompositions converge).
We can take the optimal simplicial decomposition to be integral for each $P_\delta$ by \cref{thm:simpint}.
The compactness theorem from geometric measure theory along with the above convergence result allows us to take the limit of (a subsequence of) these integral simplicial decompositions and obtain an integral flat norm decomposition of $T$ (\cref{claim:Tintdecomp}).

%We require these be chosen such that the support of $T$ is in the underlying space of $K_\delta$, $K_\delta$ is totally unimodular and every simplex of $K_\delta$ has diameter less than $\delta$.

Suppose $\delta > 0$ and apply \cref{thm:polyapprox} to obtain polyhedral currents $P_\delta$, $X_\delta$, and $S_\delta$ with 
\begin{subequations}
\label{eq:polyhedralapproximations}
\begin{align}
\label{eq:Tapprox}
\F(T-P_\delta) &< \delta, & \mass(P_\delta) &< \mass(T) + \delta, & \mass(\boundary P_\delta) & < \mass(\boundary T) + \delta,\\
\label{eq:Xapprox}
% & \mass(\boundary P_\delta) &<\mass(\boundary T) + \delta/3\\
\F(X-X_\delta) &< \delta, & \mass(X_\delta) &< \mass(X) + \delta, & \mass(\boundary X_\delta) & < \mass(\boundary X) + \delta,\\
\label{eq:Sapprox}
\F(S-S_\delta) &< \delta, & \mass(S_\delta) &< \mass(S) + \delta, & \mass(\boundary S_\delta) & < \mass(\boundary S) + \delta.
%\mass(X_\delta) &< \mass(X) + \delta/3,\\
%\mass(S_\delta) &< \mass(S) + \delta/3.
\end{align}
\end{subequations}
We also require optimal flat norm decompositions of $P_\delta - T$, $X-X_\delta$, and $S-S_\delta$, so let $U_i^\delta$, $W_j^\delta$ and $V_2^\delta$ be $d$-, $(d+1)$-, and $(d+2)$-dimensional currents such that:
\begin{subequations}
\label{eq:decompositions}
\begin{align}
\label{eq:PTdecomp}
P_\delta - T &= U_0^\delta + \boundary W_0^\delta, & \F(P_\delta - T) &= \mass(U_0^\delta) + \mass(W_0^\delta),\\
\label{eq:Xdecomp}
X - X_\delta &= U_1^\delta + \boundary W_1^\delta, & \F(X - X_\delta) &= \mass(U_1^\delta) + \mass(W_1^\delta),\\
\label{eq:Sdecomp}
S - S_\delta &= W_2^\delta + \boundary V_2^\delta, & \F(S - S_\delta) &= \mass(W_2^\delta) + \mass(V_2^\delta).
\end{align}
\end{subequations}
To clarify the notation, we adopt the convention that variables with a $\delta$ subscript are chains on the simplicial complex $K_\delta$ whereas a $\delta$ superscript merely indicates dependence on $\delta$.

Let $K_\delta$ be any simplicial complex that triangulates $P_\delta$, $X_\delta$, and $S_\delta$ separately as well as the convex hull of their union.
% such every simplex in $K_\delta$ has diameter bounded by $\delta$.
%TODO: cite a result saying this can be done.
We may assume (applying the subdivision algorithm of Edelsbrunner and Grayson\cite{EdGr2000} and \cref{thm:2dboundedsdt} if necessary) that the currents $U_0$, $U_1$, $W_0$, $W_1$, and $W_2$ can be pushed to $K_\delta$ with expansion bound at most $L$, and that the maximum diameter $\Delta$ of a simplex of $K_\delta$ satisfies
\begin{align}
\label{eq:diameter}
\Delta &\leq \frac{\delta}{\max\{1, \mass(\boundary U_0^\delta), \mass(\boundary U_1^\delta), \mass(\boundary W_0^\delta), \mass(\boundary W_1^\delta), \mass(\boundary W_2^\delta)\}}.
\end{align}

\begin{subclaim}
\label{claim:FTleq}
$\F(T) \leq \lim_{\delta \downarrow 0} \F_{K_\delta}(P_\delta)$.
\end{subclaim}
\begin{subproof}
%Let $X+\boundary S = T$ be an optimal flat norm decomposition for $T$ and $Y_\delta+\boundary R_\delta = P_\delta$ be an optimal simplicial flat norm decomposition for $P_\delta$ (that is, $\F(T) = \mass(X)+\mass(S)$ and $\F_{K_\delta}(P_\delta) = \mass(Y_\delta)+\mass(R_\delta)$)
%This is possible since TODO.
%
%Let $T-P_\delta = A_\delta + \boundary B_\delta$ be an optimal classical flat norm decomposition of the current $T-P_\delta$.
%
%The main idea here is that an optimal simplicial flat norm decomposition for $P_\delta$ can be turned into a candidate flat norm decomposition for $T$.
By the triangle inequality, and since any simplicial flat norm decomposition is a candidate decomposition for the flat norm, we have
\begin{align*}
\F(T) &\leq \F(T-P_\delta) + \F(P_\delta)\\
 &\leq \F(T-P_\delta)+\F_{K_\delta}(P_\delta).
% &= \F(T-P_\delta)+\mass(Y_\delta) + \mass(R_\delta).
\end{align*}
The claim follows from letting $\delta \downarrow 0$ and noting that $\F(T-P_\delta) \to 0$.
\end{subproof}

\begin{subclaim} $\F(T) = \lim_{\delta \downarrow 0} \F_{K_\delta}(P_\delta)$
\label{claim:FTeq}
\end{subclaim}
\begin{subproof}
In light of \cref{claim:FTleq}, we must show that $\F(T) \geq \lim_{\delta \downarrow 0} \F_{K_\delta}(P_\delta)$.
%We will show that for $\delta$ sufficiently small, the optimal flat norm decomposition for $T$ induces a simplicial flat norm decomposition for $P_\delta$ which is strictly better than our assumption allows.

%There exists $\delta^* > 0$ such that $\F_{K_\delta}(P_\delta) - \F(T) > \Delta$ for all $0 < \delta \leq \delta^*$.

\medskip
Recall that $X+\boundary S=T$ is an optimal flat norm decomposition of $T$, and that $X_\delta$ and $S_\delta$ are polyhedral approximations to $X$ and $S$ on our simplicial complex $K_\delta$.
Using the decompositions in \cref{eq:Tdecomp,eq:decompositions}, we can write
\begin{align}
\label{eq:Pdecomp}
\begin{split}
P_\delta &= T+U_0^\delta + \boundary W_0^\delta\\
&= X+\boundary S+U_0^\delta + \boundary W_0^\delta\\
&= X_\delta + U_0^\delta + U_1^\delta + \boundary(S_\delta + W_0^\delta + W_1^\delta + W_2^\delta).
\end{split}
\end{align}

\medskip
\noindent Now apply \cref{thm:2dboundedsdt} with $\epsilon = 1$ to the currents $U_i^\delta$ and $W_j^\delta$ for all $i \in \{0, 1\}$ and $j \in \{0, 1, 2\}$ to obtain $U_{i,\delta}$ and $W_{j,\delta}$ on the simplicial complex $K_\delta$ with
\begin{subequations}
\label{eq:pushdecomp}
\begin{align}
\mass(U_{i,\delta}) &\leq (11L)^{p-d+1}\mass(U_i^\delta) + (11L)^{p-d}\Delta \mass(\boundary U_i^\delta), \\
\mass(W_{j,\delta}) &\leq (11L)^{p-d}\mass(W_j^\delta) + (11L)^{p-d-1}\Delta \mass(\boundary W_j^\delta)).
\end{align}
\end{subequations}

\medskip
\noindent Applying \cref{eq:decompositions,eq:polyhedralapproximations,eq:diameter}, we obtain the following bounds from \cref{eq:pushdecomp}:
\begin{subequations}
\label{eq:pushbounds}
\begin{align}
\begin{split}
\mass(U_{i,\delta}) & \leq (11L)^{p-d+1}\delta + (11L)^{p-d}\frac{\delta}{\mass(\boundary U_i^\delta)}\mass(\boundary U_i^\delta)\\
& = (11L)^{p-d}(1 + 11L)\delta,
\end{split}\\
\begin{split}
\mass(W_{j,\delta}) & \leq (11L)^{p-d}\delta + (11L)^{p-d-1}\frac{\delta}{\mass(\boundary W_j^\delta)}\mass(\boundary W_j^\delta)\\
&= (11L)^{p-d-1}(1+11L)\delta.
\end{split}
\end{align}
\end{subequations}

\medskip
We apply the linearity result of \cref{thm:2dboundedsdt} to \cref{eq:Pdecomp} along with the fact that $P_\delta$, $X_\delta$, and $\boundary S_\delta$ are fixed by projection to the $d$-skeleton of $K_\delta$ to yield
\begin{align*}
P_\delta &=  (X_\delta + U_{0,\delta} + U_{1,\delta}) + \boundary(S_\delta + W_{0,\delta} + W_{1,\delta} + W_{2,\delta})
\end{align*}
which, as all quantities are chains on $K_\delta$, is a candidate simplicial flat norm decomposition of $P_\delta$.
Using this observation, the triangle inequality, and \cref{eq:pushbounds,eq:polyhedralapproximations}, we have
\begin{align*}
\F_{K_\delta}(P_\delta) &\leq \mass(X_\delta + U_{0,\delta} + U_{1,\delta}) + \mass(S_\delta + W_{0,\delta} + W_{1,\delta} + W_{2,\delta})\\
&\leq \mass(X_\delta) + \mass(U_{0,\delta}) + \mass(U_{1,\delta}) + \mass(S_\delta) + \mass(W_{0,\delta}) + \mass(W_{1,\delta}) + \mass(W_{2,\delta})\\
&\leq \mass(X) + \mass(S) + 2\delta + 2(11L)^{p-d}(1 + 11L)\delta + 3(11L)^{p-d-1}(1+11L)\delta\\
%&< \mass(X) + \mass(U_0^\delta) + \mass(U_1^\delta) + \mass(U_2,\delta) + \mass(S) + \mass(W_0^\delta) + \mass(W_1^\delta) + \frac{7\Delta}{12}\\
%&< \mass(X) + \mass(S) + \Delta\\
&= \F(T) + 2\delta + 2(11L)^{p-d}(1 + 11L)\delta + 3(11L)^{p-d-1}(1+11L)\delta.
\end{align*}

\medskip
\noindent The claim follows from taking the limit as $\delta \downarrow 0$.

\end{subproof}

\medskip
\begin{subclaim}
\label{claim:simpdecomp}
For each $\delta > 0$, there exist integral simplicial chains $Y_\delta$ and $R_\delta$ on $K_\delta$ such that $P_\delta = Y_\delta + \boundary R_\delta$ is an optimal simplicial flat norm decomposition (i.e., $\F_{K_\delta}(P_\delta) = \mass(Y_\delta) + \mass(R_\delta)$).
\end{subclaim}

\begin{subproof}
This result follows from \cref{thm:simpint}.
\end{subproof}

\medskip
\begin{subclaim}
\label{claim:simpdecompbounded}
There exists $c > 0$ such that for all $\delta \leq 1$, the currents $Y_\delta$, $\boundary Y_\delta$, $R_\delta$, and $\boundary R_\delta$ all have mass at most $c$.
%$There exists a $c > 0$ such that for all $\delta \leq 1$, 
%\begin{align*}
%\mass(Y_\delta) &\leq c, & \mass(\boundary Y_\delta) &\leq c,\\
%\mass(R_\delta) &\leq c, & \mass(\boundary R_\delta) &\leq c.
%\end{align*} 
\end{subclaim}
\begin{subproof}
Using the fact that $P_\delta = Y_\delta + \boundary R_\delta$ is an optimal simplicial flat norm decomposition and facts from \cref{eq:polyhedralapproximations}, we observe that
\begin{align*}
\begin{aligned}[c]
\mass(Y_\delta) &\leq \mass(P_\delta) \\
&< \mass(T) + \delta \\
&\leq \mass(T) + 1,\\
\\
\mass(R_\delta) &\leq \mass(P_\delta)\\
&< \mass(T) + \delta \\
&\leq \mass(T) + 1,\\
\end{aligned}
\qquad\qquad
\begin{aligned}
\mass(\boundary Y_\delta) &= \mass(\boundary (P_\delta - \boundary R_\delta))\\
 &= \mass(\boundary P_\delta)\\
 &< \mass(\boundary T) + \delta\\
 &\leq \mass(\boundary T) + 1,\\
\mass(\boundary R_\delta) &= \mass(P_\delta - Y_\delta)\\
&\leq \mass(P_\delta) + \mass(Y_\delta)\\
&< 2\mass(T) + 2.
\end{aligned}
\end{align*}
%\begin{align*}
%\mass(Y_\delta) &\leq \mass(P_\delta)\\
% &< \mass(T) + \delta \\
%&\leq \mass(T) + 1,\\
%\mass(R_\delta) &\leq \mass(P_\delta)\\
%&< \mass(T) + 1,\\
%\mass(\boundary Y_\delta) &= \mass(\boundary (P_\delta - \boundary R_\delta))\\
%& = \mass(\boundary P_\delta)\\
%& < \mass(\boundary T) + \delta\\
%& \leq \mass(\boundary T) + 1.\\
%\mass(\boundary R_\delta) &= \mass(P_\delta - Y_\delta)\\
%&\leq \mass(P_\delta) + \mass(Y_\delta)\\
%&< 2\mass(T) + 2 
%\end{align*}
So $c = \max\{2\mass(T)+ 2, \mass(\boundary T) + 1\}$ works.

\end{subproof}

\begin{subclaim}
\label{claim:Tintdecomp}
There is an optimal flat norm decomposition of $T$ with integral currents.
\end{subclaim}
\begin{subproof}
The compactness theorem\cite{Federer1969,Morgan2008} states that given any closed ball $B$ in $\R^n$ and nonnegative constant $c$, the set 
\[
\{I \text{ is an integral $p$-current in $\R^n$} \mid \mass(I) \leq c, \mass(\boundary I) \leq c, \spt I \subseteq B\}
\]
 is compact with respect to the flat norm.
In light of \cref{claim:simpdecompbounded}, this means there is a compact set of integral currents containing $Y_\delta$ for all $\delta \leq 1$ (and similarly for $R_\delta$).

Let $\delta_n = \frac{1}{n}$ and consider the sequences $\{Y_{\delta_n}\}$ and $\{R_{\delta_n}\}$.
By compactness, there exists a subsequence $\{\delta^*_n\}$ of $\{\delta_n\}$ and integral currents $Y^*$ and $R^*$ such that $Y_{\delta^*_n} \to Y^*$ and $R_{\delta^*_n} \to R^*$ in the flat norm.
By \cref{lem:convprops}, we have $Y_{\delta^*_n} + \boundary R_{\delta^*_n} \to Y^* + \boundary R^*$.
Applying \cref{claim:simpdecomp} and \cref{claim:FTeq}, we get
\begin{align}
\label{eq:subseqflatconv}
\mass(Y_{\delta^*_n}) + \mass(R_{\delta^*_n}) = \F_{K_{\delta^*_n}} (P_{\delta^*_n}) \to \F(T).
\end{align}

Since $Y_\delta + \boundary R_\delta = P_\delta \to T$, we also have $Y_{\delta^*_n} + \boundary R_{\delta^*_n} \to T$.
That is, $T = Y^* + \boundary R^*$.
As mass is lower semicontinuous with respect to convergence in the flat norm and by \cref{claim:FTeq}, we have that 
\begin{align}
\label{eq:limitattainsflatnorm}
\begin{split}
\mass(Y^*) + \mass(R^*) &\leq \lim_{n\to\infty} \mass(Y_{\delta^*_n}) + \mass(R_{\delta^*_n}) \\
&= \lim_{n\to\infty} \F_{K_{\delta^*_n}}(P_{\delta^*_n})\\
&= \F(T).
\end{split}
\end{align}
Thus $\mass(Y^*) + \mass(R^*) = \F(T)$ and $Y^* + \boundary R^*$ is an optimal flat norm decomposition of $T$.
\end{subproof}
\end{proof}

We restate our main result for the cases where \cref{con:boundreg} is known to hold (see \cref{thm:2dboundreg}), again emphasizing that progress on the conjecture extends this result.

\begin{corollary}
If $T$ is an integral $1$-current in $\R^{2}$, then some flat norm minimizer for $T$ is an integral current.
That is, there is an integral $1$-current $X_I$ and integral $2$-current $S_I$ such that $\F(T) = \mass(X_I) + \mass(S_I)$ and $T = X_I + \boundary S_I$.
\end{corollary}

\section{Flat Norm Decomposition and the Least Area Problem}
\label{sec:leastarea}
Given a current $T$ with $\boundary T = 0$, the least area problem is to find a current $S_T$ with minimal mass such that $\boundary S_T = T$ (this corresponds to a flat norm decomposition for which $X = T - \boundary S$ is constrained to be empty).
Even if $T$ is integral, it is not necessary that the least area be attained by an integral $S_T$.
Indeed, Young\cite{Yo1963}, White\cite{Wh1984}, and Morgan\cite{Mo1984} provide examples where the minimizers for a given integral boundary must be nonintegral.
We show that these imply similar examples for the flat norm problem.
In particular, the flat norm decomposition of integral currents cannot be taken to be integral in general for codimension 3 or higher.

It is notationally convenient for us to add a scaling parameter to the flat norm as suggested by previous work\cite{MoVi2007}.
In particular, we let $\F_\lambda$ denote the flat norm with scale $\lambda$ and define it by
\[
\F_\lambda(T) = \min_S \left(M(T-\boundary S) + \lambda M(S)\right).
\]
Although we prove results using $\F_\lambda$, we note that equivalent statements hold for the unscaled flat norm.
This follows from the correspondence $\F_\lambda(T) = \lambda^{-m} \F(T_\lambda)$ for an $m$-current $T$ and $\lambda > 0$ where $T_\lambda$ is a $\lambda$-dilation of $T$\cite{MoVi2007}.

\begin{lemma}
Suppose $\boundary T = 0$ and let $S_T$ denote a minimal mass current with boundary $T$.
For sufficiently small $\lambda$, we have $\F_\lambda(T) = \lambda \mass(S_T) < \mass(T-\partial S) + \lambda \mass(S)$ for any $S$ such that $T-\partial S \neq 0$. That is, it is strictly better to span T than not to span T when searching for a minimal flat norm decomposition of T.
\end{lemma}
\begin{proof}
The isoperimetric inequality\cite[5.3]{Morgan2008} for integral $m$-dimensional boundaries in $\R^n$ gives us
\begin{align}
\label{eq:isoperimetricbound}
\mass(S_B) \leq \alpha \mass(B)^{\frac{m+1}{m}}
\end{align}
where $S_B$ is a minimal mass current with boundary $B$ and $\alpha$ is independent of $B$.
Now choose 
\begin{align}
\label{eq:lambdabound}
\lambda < \min\cbr{1, \frac{1}{\alpha \mass(S_T)^{\frac{1}{m}}}}.
\end{align}
Let $\tilde{S}$ be part of an optimal $\lambda$-flat norm decomposition of $T$.
That is, we have $T = (T-\boundary \tilde{S}) + \boundary \tilde{S}$ and 
\begin{align}
\label{eq:flatdecomp}
\F_\lambda(T) = \mass(T-\boundary \tilde{S}) + \lambda\mass(\tilde{S}).
\end{align}
Assume also that $\mass(T-\boundary \tilde{S}) > 0$. Using \cref{eq:lambdabound,eq:flatdecomp} along with the fact that $\boundary S_T = T$ is a decomposition of $T$, we obtain
\begin{align*}
\mass(T-\boundary \tilde{S}) &\leq \F_\lambda(T)\\
&\leq \lambda\mass(S_T)\\
&< \mass(S_T)
\end{align*}
which implies
\begin{align}
\label{eq:TboundarySexponentbound}
\mass(T-\boundary \tilde{S})^{\frac{m+1}{m}} &< \mass(T-\boundary \tilde{S}) \mass(S_T)^{\frac{1}{m}}.
\end{align}
Applying \cref{eq:isoperimetricbound,eq:lambdabound,eq:TboundarySexponentbound} we get
\begin{align}
\label{eq:minimalsurfacebound}
\begin{split}
\lambda \mass(S_{T-\boundary \tilde{S}}) &< \lambda \alpha \mass(T-\boundary \tilde{S}) \mass(S_T)^{\frac{1}{m}}\\
&< \mass(T-\boundary \tilde{S}).
\end{split}
\end{align}
Since $\boundary(S_{T-\boundary \tilde{S}} + \tilde{S}) = T-\boundary \tilde{S} + \boundary S = T$, its mass is no smaller than that of $S_T$, the minimal mass surface with this boundary.
Using this fact along with the triangle inequality and \cref{eq:minimalsurfacebound,eq:flatdecomp}, we have:
\begin{align*}
\begin{split}
\F_\lambda(T) &\leq \lambda\mass(S_T)\\
              &\leq {\lambda} \mass(S_{T-\boundary \tilde{S}}  + \tilde{S}) \\
              &\leq {\lambda} \mass(S_{T-\boundary \tilde{S}}) + {\lambda} \mass(\tilde{S}) \\ 
              &< \mass(T-\boundary \tilde{S}) + {\lambda} \mass(\tilde{S})\\
              &= \F_\lambda(T).
\end{split}
\end{align*}
This is a contradiction. Thus $\mass(T-\boundary \tilde{S}) =0$ and we conclude that for sufficiently small $\lambda$, $\F_\lambda(T) = \lambda \mass(S_T) < \mass(T-\partial S) + \lambda \mass(S)$ for all $S$ such that $T-\partial S \neq 0$.
\end{proof}

\begin{corollary}
\label{cor:leastareaexamplesareflatnormexamples}
The examples of Young\cite{Yo1963}, White\cite{Wh1984}, and Morgan\cite{Mo1984} (demonstrating integral boundaries for which the least area problem minimizer is nonintegral) directly imply the existence of integral currents where, for some $\lambda > 0$, every optimal $\F_\lambda$ decomposition is nonintegral (simply choose $\lambda$ small enough and/or scale up the example sufficiently).
\end{corollary}

\begin{corollary}
Given $d \geq 1$ and $n \geq d + 3$, there exist integral $d$-currents in $\R^n$ without any integral optimal flat norm decompositions.
\end{corollary}
\begin{proof}
For concreteness, we concentrate on Morgan's least area example which can construct an integral $d$-current in $\R^{d+3}$ without an integral optimal decomposition\cite[Remark 1.2]{Mo1984}.
% or using \cite[5.4.9]{Federer1969} and facts about slicing.
This can then be trivially embedded in $\R^n$.
\end{proof}

\section{Discussion}
Although we specifically show integral optimal decompositions exist only for integral $1$-currents in $\R^2$, our framework provides a clear way to extend these results to $d$-currents in $\R^{d+1}$: find a higher-dimensional subdivision algorithm that bounds our simplicial regularity constants overall and isolates the problematic portions appropriately (\cref{con:boundreg}).
Similarly, if the class of simplicial complexes we generate could be forced to be free of relative torsion (and hence their boundary matrices are totally unimodular\cite{DeHiKr2011}) in higher codimensions, we could extend our results to, for example, $1$-currents in $\R^3$.
However, in light of counterexamples in codimension $3$ to the related least spanning area problem\cite{Yo1963,Wh1984,Mo1984}, the results cannot be extended to $1$-currents in $\R^d$ for $d \geq 4$.

In two dimensions, we used the minimum angle of the triangular complex as a surrogate to bound our simplicial regularity constant.
Unfortunately, dihedral angles cannot serve the same role for tetrahedra on account of ``spires'', which are irregular simplices with nice dihedral angles\cite[Section 1.7]{ChDeSh2012}.
One promising idea is to use densities, another tool from geometric measure theory.
Essentially, the density of a simplex at a vertex is the ratio of a (sufficiently small) vertex-centered sphere's volume inside the simplex to the sphere's total volume.

Shifting out of the simplicial setting, we see no obvious impediments to reworking this result in terms of cellular complexes and using a version of Sullivan's cellular deformation theorem\cite{Sullivan1990} modified in the same way as our multiple current simplicial deformation theorem.
For this approach to work, we need to also verify the simplicial integrality results\cite{DeHiKr2011,IbKrVi2013} in the cellular setting.

This approach has the possible advantage that the conjectured subdivision algorithm need not generate a simplicial complex; for example, in two dimensions we would be finished after superimposing the square grid in \cref{thm:2dboundreg}.
This weaker restriction may be helpful when working in higher dimensions.

For another approach, we note that the subdivision algorithm is only used to prove \cref{claim:FTeq} (that $\lim_{\delta \downarrow 0} \F_{K_\delta}(P_\delta) = \F(T)$), which is then used in \cref{claim:Tintdecomp} to show that the limit of the simplicial decompositions is an optimal decomposition for $T$.
\cref{claim:FTeq} is sufficient to reach this conclusion, but not necessary: \cref{fig:flatnormconv} provides an example where \cref{claim:FTeq} does not hold but the limit of the simplicial decompositions is the optimal decomposition for $T$.
Could we show the limit is an optimal decomposition of $T$ for a more general class of simplicial decompositions?

Our approach using results on triangulations and simplicial complexes is novel compared to the traditional methods from geometric measure theory typically used to tackle problems similar to the one we have considered.
Are there other questions in geometric analysis that could be answered (more efficiently) using approaches based on computational geometry?
One potential candidate could be the problem of counting the number of minimal surfaces of a specified topological type bounded by a given curve\cite{HoWh2008}.

\bibliography{flatnorm,homology}
\bibliographystyle{plain}

\end{document}